\documentclass[article, 11pt]{amsart}
\usepackage{tikz}
\usepackage{amsfonts,amsthm,amsbsy,amsmath,amssymb,verbatim}
\usepackage{latexsym,enumerate,enumitem}
\usepackage{mathrsfs} 
\usepackage{enumitem}
\usepackage{tikz-cd}
\usetikzlibrary{positioning}
\usepackage{xcolor}
\usepackage{setspace}
\allowdisplaybreaks
\usepackage{geometry}
\usepackage{hyperref}
	\hypersetup{colorlinks, breaklinks,
            	linkcolor = blue,
				urlcolor = blue,
				anchorcolor = blue,
				citecolor = blue}

\usepackage{cleveref}

\usepackage{caption}
\usepackage{subcaption}

\newtheorem{theorem}{Theorem}

\newtheorem{lemma}[theorem]{Lemma}
\newtheorem{prop}[theorem]{Proposition}

\newtheorem*{theorem*}{Theorem}
\newtheorem*{lemma*}{Lemma}
\newtheorem*{prop*}{Proposition}
\newtheorem*{corollary*}{Corollary}
\newtheorem*{remark*}{Remark} 
\newtheorem*{remarks*}{Remarks}
\newtheorem*{conj*}{Conjecture}

\def\S{\mathbb{S}}

\def\R{{\mathbb R}}
\def\Z{{\mathbb Z}}

\def\s{\sigma}

\newcommand{\supp}{\mathrm{supp}}

\DeclareFontFamily{U}{mathx}{\hyphenchar\font45}
\DeclareFontShape{U}{mathx}{m}{n}{
	<5> <6> <7> <8> <9> <10>
	<10.95> <12> <14.4> <17.28> <20.74> <24.88>
	mathx10
}{}

\def\wh{\widehat}

\newtheorem{innercustomgeneric}{\customgenericname}
\providecommand{\customgenericname}{}
\newcommand{\newcustomtheorem}[2]{%
	\newenvironment{#1}[1]
	{%
		\renewcommand\customgenericname{#2}%
		\renewcommand\theinnercustomgeneric{##1}%
		\innercustomgeneric
	}
	{\endinnercustomgeneric}
}

\numberwithin{equation}{section}

\newcustomtheorem{customthm}{Theorem}
\newcustomtheorem{customlemma}{Lemma}
\newcustomtheorem{customprop}{Proposition}

\begin{document}

 \title[Bilinear singular integral operators with kernels in weighted spaces]{Bilinear singular integral operators with kernels in weighted spaces}
\author[P.~Honz{\'i}k, S.~Lappas and L.~Slav{\'i}kov{\'a}]{Petr~Honz{\'i}k, Stefanos~Lappas and Lenka~Slav{\'i}kov{\'a}}


\newcommand{\Addresses}{{
		\bigskip
		\footnotesize

        \textsc{Petr~Honz{\'i}k.}
		\textsc{Department of Mathematical Analysis, Faculty of Mathematics and Physics, Charles University, Sokolovsk\'a 83, 186 75 Praha 8, Czech Republic}\par\nopagebreak
		\textit{E-mail address:} \texttt{honzikpe@karlin.mff.cuni.cz}

		\textsc{Stefanos~Lappas.}
		\textsc{Department of Mathematical Analysis, Faculty of Mathematics and Physics, Charles University, Sokolovsk\'a 83, 186 75 Praha 8, Czech Republic}\par\nopagebreak
		\textit{E-mail address:} \texttt{stefanos.lappas@matfyz.cuni.cz}
		
		\textsc{Lenka~Slav{\'i}kov{\'a}.}
		\textsc{Department of Mathematical Analysis, Faculty of Mathematics and Physics, Charles University, Sokolovsk\'a 83, 186 75 Praha 8, Czech Republic}\par\nopagebreak
		\textit{E-mail address:} \texttt{slavikova@karlin.mff.cuni.cz}

}}

\begin{abstract} 
We establish the full quasi-Banach range of $L^{p_1}(\R) \times L^{p_2}(\R) \rightarrow L^p(\R)$ bounds for one-dimensional bilinear singular integral operators with homogeneous kernels whose restriction $\Omega$ to the unit sphere $\S^1$ is supported away from the degenerate line $\theta_1=\theta_2$, belongs to $L^q(\S^1)$ for some $q>1$ and has vanishing integral. In fact, a more general result is obtained by dropping the support condition on $\Omega$ and requiring that $\Omega\in L^q(\S^1,u^q)$, where $u(\theta_1,\theta_2)=|\theta_1-\theta_2|^{-1}$ for $(\theta_1,\theta_2)\in \S^1$. In addition, we provide counterexamples that show the failure of the $n$-dimensional version of the previous result when $n\geq 2$, as well as the failure of its $m$-linear variant in dimension one when $m\geq 3$. The relationship of these results to (un)boundedness properties of higher-dimensional multilinear Hilbert transforms is also discussed.
\end{abstract}

\maketitle

\section{Introduction}

The directional bilinear Hilbert transforms are operators of the form
\begin{equation}\label{E:hilbert}
T_\theta(f_1,f_2)(x)=\operatorname{p.v.} \int_{\R} f_1(x-t\theta_1) f_2(x-t\theta_2) \frac{dt}{t}, \quad x\in \mathbb R,
\end{equation}
where $\theta=(\theta_1,\theta_2) \in \mathbb R^2$. The question of their uniform boundedness in the parameter $\theta$ was formulated by Calder\'on in the 1960s as a potential way to prove the boundedness of the Calder\'on first commutator and, in turn, of the Cauchy integral on Lipschitz curves. While bounds for the latter two operators are well understood by now~\cite{C65}, \cite{CMM}, and the same is true for the uniform $L^{p_1}(\R) \times L^{p_2}(\R) \rightarrow L^{p}(\R)$ bounds for the operators~\eqref{E:hilbert} when 
\begin{equation}\label{E:range-bht}
1<p_1,p_2<\infty, \quad \frac{2}{3}<p<\infty \quad \text{and} \quad \frac{1}{p}=\frac{1}{p_1}+\frac{1}{p_2},
\end{equation}
see~\cite{UW},
the problem of determining the full quasi-Banach range of boundedness and uniform boundedness of~\eqref{E:hilbert} remains largely open. In this article, we focus on the simpler but closely related question of boundedness of bilinear singular integral operators obtained by averaging the operators~\eqref{E:hilbert} over all $\theta$ from the unit sphere $\S^1$ with respect to a very rough weight $\Omega$. We show that such operators are bounded in the full quasi-Banach range of exponents \begin{equation}\label{E:range-full}
1<p_1,p_2<\infty, \quad \frac{1}{2}<p<\infty \quad \text{and} \quad \frac{1}{p}=\frac{1}{p_1}+\frac{1}{p_2},
\end{equation}
as long as $\Omega$ is supported away from the degenerate direction $\theta_1=\theta_2$.

Assume, for simplicity, that $\Omega$ is odd and integrable over $\mathbb S^1$. Then the averaged operator
\begin{equation}\label{E:rotation}
T_\Omega(f_1,f_2)(x)=\frac{1}{2}\int_{\S^{1}} \Omega(\theta) T_\theta(f_1,f_2)(x)\,d\theta, \quad x\in \mathbb R,
\end{equation}
can be rewritten as 
\begin{equation}\label{eq:T}
T_\Omega(f_1,f_2)(x) = \operatorname{p.v.} \int_{\R^{2}} \frac{\Omega((y_1,y_2)/|(y_1,y_2)|)}{|(y_1,y_2)|^{2}} f_1(x-y_1)f_2(x-y_2) dy, \quad x\in \R,
\end{equation}
see~\cite[pages 162 -- 163]{GT}. We note that the boundedness results for the operator $T_\Omega$ often remain valid when the assumption that $\Omega$ is odd is replaced by the weaker condition that $\Omega$ has vanishing integral over $\S^1$. In this situation, however, $T_\Omega$ can no longer be expressed in the form~\eqref{E:rotation}. 

Before stating our main result in detail, we point out some aspects of the problem which motivate our choice of the functional framework for the function $\Omega$. We recall that the generic case of the directional bilinear Hilbert transform~\eqref{E:hilbert} is the one when 
\begin{equation}\label{E:theta-range}
\theta_1\neq 0, \quad \theta_2 \neq 0 \quad \text{and}  \quad \theta_1\neq \theta_2.
\end{equation} 
If all these conditions are fulfilled, then~\eqref{E:hilbert} is a genuinely bilinear operator, whose $L^{p_1}(\R) \times L^{p_2}(\R) \rightarrow L^{p}(\R)$-boundedness in the range of exponents~\eqref{E:range-bht} was established by Lacey and Thiele~\cite{LT97}, \cite{LT99} via time-frequency analysis techniques, and the corresponding uniform bounds were subsequently proved in~\cite{Thi02}, \cite{GL}, \cite{Li}, \cite{UW}. On the other hand, if one of the conditions~\eqref{E:theta-range} fails, then the operator~\eqref{E:hilbert} becomes degenerate and its boundedness properties can be deduced from the boundedness properties of the standard (linear) Hilbert transform. In particular, when $\theta_1=\theta_2$ then the operator $T_\theta$ reduces to the Hilbert transform applied to the pointwise product of the input functions. Since the Hilbert transform is unbounded on $L^p(\R)$ when $p\leq 1$, the operator $T_\theta$ cannot be $L^{p_1}(\R) \times L^{p_2}(\R) \rightarrow L^{p}(\R)$-bounded for $p\leq 1$ in this degenerate situation. To summarize, we have observed that the boundedness properties of the operator $T_\theta$ depend on the particular choice of $\theta$. 

When $\Omega \in L^\infty(\S^1)$, then the $L^{p_1}(\R) \times L^{p_2}(\R) \rightarrow L^{p}(\R)$-boundedness of the operator $T_\Omega$ in the full quasi-Banach range of exponents~\eqref{E:range-full} was established in~\cite{GHH18}. Analogous results in the case when $\Omega$ possesses some amount of smoothness were available before, see~\cite{CM75}, \cite{GT}, \cite{KS99}. The authors of~\cite{GHH18} employed a new technique based on decomposing the rough kernel of the operator $T_\Omega$ into a series of smooth kernels, and then decomposing each smooth kernel on the frequency side in terms of product-type compactly supported smooth wavelets. While the condition $\Omega \in L^\infty(\S^1)$ is weak enough to make the averaging procedure in~\eqref{E:rotation} rough, it still has enough smoothening effect to ensure 
that the unboundedness of the operator $T_{(\theta_1,\theta_1)}$ outside of the Banach range of exponents does not cause the unboundedness of the averaged operator. In the present article, we focus on the case when $\Omega \in L^q(\S^1)$ for $q>1$. In this setting, the effects of the degenerate operators $T_{(\theta_1,\theta_1)}$ on the boundedness of the averaged operator $T_\Omega$ can be clearly seen and require us to modify the functional framework for the function $\Omega$ in a way which has not been considered in earlier articles on this topic. Our first main result has the following form. 

\begin{theorem}\label{T:main-result-restricted}
Let $1<p_1,p_2<\infty$, $1/2<p<\infty$ satisfy $1/p=1/p_1+1/p_2$, and let $q>1$ and $\alpha>0$. Assume that $\Omega \in L^q(\mathbb S^1)$ fulfills $\int_{\S^{1}} \Omega(\theta) d\sigma(\theta)=0$ and is supported in the set of those $(\theta_1,\theta_2)\in \S^1$ for which $|\theta_1-\theta_2|\geq \alpha$. Then there exists a constant $C=C(p_1,p_2,q,\alpha)$ such that
\begin{equation}\label{E:boundedness-restricted-support}
\big\Vert T_{\Omega}(f_1,f_2)\big\Vert_{L^{p}(\R)}\leq C\Vert\Omega\Vert_{L^q(\S^{1})}\Vert f_1\Vert_{L^{p_1}(\R)}\Vert f_2\Vert_{L^{p_2}(\R)}.
\end{equation}
\end{theorem}

Theorem~\ref{T:main-result-restricted} fails in the absence of the support condition on $\Omega$. This follows from the results of~\cite{DS23} (see also~\cite{DGHST11}, \cite{GHH18}, \cite{GHS20}, \cite{HP21} for earlier partial results), which imply that in such a situation, inequality~\eqref{E:boundedness-restricted-support} holds whenever $1/p+1/q<2$ but fails otherwise. In fact, we also prove a more general variant of Theorem~\ref{T:main-result-restricted} in which the support condition on $\Omega$ is replaced by the weaker requirement that $\Omega \in L^q(\S^1,u^q)$ for a suitable weight $u$ which behaves like a constant away from the degenerate points $\pm (1/\sqrt{2},1/\sqrt{2})$ and blows up at these points. The corresponding result is stated under the assumption that $1/p+1/q\geq 2$. This causes no loss of generality as the case $1/p+1/q<2$ was treated in~\cite{DS23}. In particular, since $q>1$, we always have $p<1$ in the result below.

\begin{theorem}\label{thm:main result}
Let $1<p_1,p_2<\infty$, $1/2<p<\infty$ satisfy $1/p=1/p_1+1/p_2$, and let $q>1$ be such that $1/p+1/q\geq 2$. 

\noindent
\textup{(i)} Given $\varepsilon>0$, we set 
\[
u(\theta_1,\theta_2)=|\theta_1-\theta_2|^{2-\frac{1}{p}-\frac{1}{q}-\varepsilon} \quad \text{for} \quad (\theta_1,\theta_2)\in \S^1.
\]
Suppose that $\Omega \in L^q(\S^{1},u^q)$ and $\int_{\S^{1}} \Omega(\theta) d\sigma(\theta)=0$. Then there exists a constant $C=C(p_1,p_2,q,\varepsilon)$ such that

\begin{equation*}
\big\Vert T_{\Omega}(f_1,f_2)\big\Vert_{L^{p}(\R)}\leq C\Vert\Omega\Vert_{L^q(\S^{1},u^q)}\Vert f_1\Vert_{L^{p_1}(\R)}\Vert f_2\Vert_{L^{p_2}(\R)}.
\end{equation*}

\noindent
\textup{(ii)}
Let \[
u(\theta_1,\theta_2)=|\theta_1-\theta_2|^{2-\frac{1}{p}-\frac{1}{q}} \quad \text{for} \quad (\theta_1,\theta_2)\in \S^1.
\]
Then 
there exist an odd function $\Omega \in L^q(\S^{1},u^q)$ and $f_i\in L^{p_i}(\R)$, $i=1,2$, such that $T_{\Omega}(f_1,f_2)\notin L^{p}(\R)$.
\end{theorem}

Theorem~\ref{T:main-result-restricted} is a direct consequence of Theorem~\ref{thm:main result}, and it thus suffices to prove the latter. Our proof builds on the following geometric observation. Suppose that $(\theta_1,\theta_2)$ is a point from the unit sphere $\S^1$ and that $f_1$ and $f_2$ are formally chosen to be the Dirac delta measures at points $\theta_1$ and $\theta_2$, respectively. Then 
\begin{equation*}
T_\Omega(f_1,f_2)(x)=\frac{\Omega \left(\frac{(x-\theta_1,x-\theta_2)}{|(x-\theta_1,x-\theta_2)|}\right)}{|(x-\theta_1,x-\theta_2)|^{2}}, \quad x\in \R.
\end{equation*}
As long as $\theta_1$ is far away from $\theta_2$, the mapping
\begin{equation}\label{E:mapping}
x\mapsto \frac{(x-\theta_1,x-\theta_2)}{|(x-\theta_1,x-\theta_2)|},
\end{equation}
which amounts to shifting the point $(-\theta_1,-\theta_2)\in \S^1$ by the vector $(x,x)$ and projecting it back to the unit sphere, maps a small interval centered at the origin onto an arc on the unit sphere of comparable length. However, as $\theta_1-\theta_2$ approaches $0$, the arc becomes smaller until it finally degenerates to a point once $\theta_1=\theta_2$, see Figure~\ref{figure} for an illustration of this process. 
Such degeneracies in turn prevent bounding the norm of $T_\Omega(f_1,f_2)$ from above by the $L^q$-norm of $\Omega$ when $q<\infty$.

\begin{center}
\begin{figure}[htb]
\begin{tikzpicture}

\draw (0,0) circle (1);

\draw[red,thick,->] (0,1) -- (1/2,3/2);
\draw[red,thick,->] (0.4,0.9165) -- (0.9,1.4165);
\draw[red,thick,->] (-0.7071,0.7071) -- (-0.2071,1.2071);
\draw[red,thick,->] (0.7071,0.7071) -- (1.2071,1.2071);

\draw[dashed] (0,0) -- (1/2,3/2);
\draw[dashed] (0,0) -- (0.9,1.4165);
\draw[dashed] (0,0) -- (-0.2071,1.2071);
\draw[dashed] (0,0) -- (0.7071,0.7071);

\draw[violet, thick,->] (0,1) arc[start angle=90, end angle=atan(3/sqrt(10)/(1/sqrt(10))), radius=1];
\draw[violet, thick,->] (0.4,0.9165) arc[start angle=atan(0.9165/0.4), end angle=atan(1.4165/0.9), radius=1];
\draw[violet, thick,->] (-0.7071,0.7071) arc[start angle=135, end angle=atan(-1.2071/0.2071)+180, radius=1];
\end{tikzpicture}
\caption{Behavior of~\eqref{E:mapping} for different values of $(\theta_1,\theta_2)$.}
 \label{figure}
\end{figure}
\end{center}

In order to turn the above-mentioned geometric observation into a formal proof of Theorem~\ref{thm:main result}, we perform various decompositions of the kernel of the operator $T_\Omega$ and of the input functions $f_1$ and $f_2$. Specifically, we first express the rough kernel as an infinite sum of smooth kernels, following the approach of Duoandiko\-etxea and Rubio de Francia~\cite{DuRu1986}, which was pioneered in the bilinear setting in~\cite{GHH18}. To prove bounds for each of the operators with smooth kernels, we perform a Calder\'on-Zygmund decomposition of the functions $f_1$ and $f_2$. The ``good'' part of the operator is then estimated using the results of~\cite{DS23} while estimates for the ``bad'' part of the operator employ the above-mentioned geometric argument. 
Since the Calder\'on-Zygmund decomposition does not yield sufficiently strong quantitative estimates of the norms of the operators with smooth kernels, the results are then interpolated with those from~\cite{DS23} to obtain the desired estimates. We combine standard multilinear interpolation theorems~\cite{Ja1988} with the recent results of Cao, Olivo and Yabuta~\cite{COY22}, involving multilinear interpolation in the setting of weighted Lebesgue spaces. 

We will next comment on some of the difficulties which arise when trying to extend the results of Theorems~\ref{T:main-result-restricted} and~\ref{thm:main result} to the higher-dimensional and/or multilinear setting. The directional $m$-linear Hilbert transforms in $n$ dimensions are operators of the form
\begin{equation*}
T_\theta(f_1,\dots,f_m)(x)=\operatorname{p.v.} \int_{\R} f_1(x-t\theta_1) \cdots f_m(x-t\theta_m) \frac{dt}{t}, \quad x\in \mathbb R^n,
\end{equation*}
where $\theta=(\theta_1,\dots,\theta_m) \in \mathbb R^{mn}$ and $f_1,\dots,f_m$ are $n$-dimensional input functions. No bounds for these operators are known when $m\geq 3$ or $n\geq 2$. In fact, counterexamples exist, showing that the generic case of these operators becomes unbounded in a part of the quasi-Banach range of exponents; see~\cite{Demeter} for the case $m=3$ and $n=1$, and~\cite{KTZ} for $m=n=2$. Despite of this fact, the averaged operator 
\begin{equation}\label{E:rotation-multilinear}
T_\Omega(f_1,\dots,f_m)(x)=\frac{1}{2}\int_{\S^{mn-1}} \Omega(\theta) T_\theta(f_1,\dots,f_m)(x)\,d\theta, \quad x\in \mathbb R^n,
\end{equation}
is $L^{p_1}(\R^n) \times \dots \times L^{p_m}(\R^n) \rightarrow L^p(\R^n)$-bounded in the full quasi-Banach range of exponents
\begin{equation}\label{E:quasi-banach-multilinear}
1<p_1,\dots,p_m<\infty, \quad \frac{1}{m}<p<\infty \quad \text{and} \quad \frac{1}{p}=\frac{1}{p_1}+\dots+\frac{1}{p_m},
\end{equation}
as long as $\Omega$ is an odd function on the unit sphere $\mathbb S^{mn-1}$ which is sufficiently regular. When $\Omega$ is Lipschitz, such a result follows from the multilinear Calder\'on-Zygmund theory~\cite{GT}, while the analogous result for $\Omega \in L^\infty(\S^{mn-1})$ was established in~\cite{GHHP22}. We point out that, as previously, the operator~\eqref{E:rotation-multilinear} can be rewritten as
\begin{equation}\label{eq:T-multilinear}
T_\Omega(f_1,\dots,f_m)(x) = \operatorname{p.v.} \int_{\R^{mn}} \frac{\Omega((y_1,\dots,y_m)/|(y_1,\dots,y_m)|)}{|(y_1,\dots,y_m)|^{mn}} f_1(x-y_1) \cdots f_m(x-y_m) dy
\end{equation}
for $x\in \R^n$, see~\cite[pages 162 -- 163]{GT}.

The operator $T_\Omega$ is no longer bounded in the full quasi-Banach range of exponents~\eqref{E:quasi-banach-multilinear} when $\Omega$ merely belongs to  $L^q(\S^{mn-1})$ for $q<\infty$. When no support conditions on the function $\Omega$ are imposed, such counterexamples follow from~\cite{GHHP22}. We show that if $m\geq 3$ or $n\geq 2$, counterexamples can be constructed even under the nondegeneracy assumptions as in Theorem~\ref{T:main-result-restricted}. This provides a further indication that, in contrast to the case $q=\infty$, the (un)boundedness properties of the averaged operator $T_\Omega$ with $\Omega \in L^q(\S^{mn-1})$ for $q$ close to $1$ reflect well the (un)boundedness properties of the operators $T_\theta$.

\begin{theorem}\label{T:counterexample}
\textup{(i)}
Let $n\geq 2$, $1<p_1,p_2<\infty$, $1/p=1/p_1+1/p_2$ and let $q\in (1,\infty)$ be such that
\begin{equation}\label{prangebili}
p< \frac{nq}{1-n+2nq}.
\end{equation}
Then there exists an odd function $\Omega\in L^q(\mathbb S^{2n-1})$ supported in the set 
\begin{equation*}
\{(\theta_1,\theta_2)\in \mathbb S^{2n-1}:~ |\theta_1-\theta_2|>1/2\}
\end{equation*}
such that the operator $T_\Omega$ is not bounded as bilinear operator from $L^{p_1}(\mathbb R^n)\times L^{p_2}(\mathbb R^n)$ to $L^p(\mathbb R^n).$

\textup{(ii)}
Let $m\geq 3,$ $1<p_1,\dots,p_m<\infty$, $1/p=1/p_1+\cdots+1/p_m$ and let $q\in (1,\infty)$ be such that 
\begin{equation}\label{prangemultili}
p<\frac{q}{(q-1)m +2}.
\end{equation}
Then there exists an odd function $\Omega\in L^q(\mathbb S^{m-1})$ supported in the set 
\begin{equation*}
\{(\theta_1,\theta_2,\dots,\theta_m)\in \mathbb S^{m-1}:~ |\theta_i-\theta_j|>m^{-3/2} \text{ for } i\neq j\}
\end{equation*}
such that the operator $T_\Omega$ is not bounded as $m$-linear operator from $L^{p_1}(\mathbb R)\times\cdots\times L^{p_m}(\mathbb R)$ to $L^p(\mathbb R).$
\end{theorem}

We note that in the bilinear setting, the set of those $q\in (1,\infty)$ satisfying condition~\eqref{prangebili} is nonempty whenever $\frac{1}{2}<p< \frac{n}{n+1}$,  while in the multilinear case in dimension one, our counterexamples imply unboundedness in a nontrivial range of $q'$s whenever $\frac{1}{m}<p<\frac 1 2$.

Our proof of Theorem~\ref{T:counterexample} makes use of geometric constructions which are higher-dimensional variants of those depicted in Figure~\ref{figure}. In particular, the crucial observation is that for any $(\theta_1,\dots,\theta_m)\in \S^{mn-1}$, the mapping
\begin{equation*}
x\mapsto \frac{(x-\theta_1,\dots,x-\theta_m)}{|(x-\theta_1,\dots,x-\theta_m)|}
\end{equation*}
maps a small ball in $\R^n$ centered at the origin onto a subset of the unit sphere of $(mn-1)$-dimensional measure zero. 
It seems likely that in order to prove positive results in the spirit of Theorems~\ref{T:main-result-restricted} and~\ref{thm:main result} in the case when $m\geq 3$ or $n\geq 2$, one would have to modify the averaging procedure which led to the definition of the operator $T_\Omega$ in~\eqref{E:rotation-multilinear}. Namely, rather than averaging operators with one-dimensional kernels over an $(mn-1)$-dimensional set of directions, one would have to average operators with higher-dimensional (and sufficiently regular) kernels over a lower-dimensional set of directions. This hypothesis is supported by the fact that the known quasi-Banach estimates for the bilinear Hilbert transform can be understood within a more general context of quasi-Banach estimates for certain multilinear operators with higher-dimensional kernels~\cite{MTT}.
A more thorough discussion of this situation is however outside of the scope of the present paper.

\medskip
The paper is structured as follows. In Section~\ref{S:overview}, we provide an overview of the proof of Theorem~\ref{thm:main result}, while the specific steps of the proof are carried out in Sections~\ref{S:CZ decomposition}--\ref{S:counterexamle n=1}. Theorem~\ref{T:counterexample} is proved in Section~\ref{S:counterexamle-higher-n}.

\subsection*{Notation}
Throughout the paper, the symbol $f\lesssim g$ means that $f\leq Cg$ for some constant $C>0$ that may vary from line to line and depends on the main parameters, which will often be specified explicitly. We write $f\approx g$ if $f\lesssim g$ and $g\lesssim f$ hold simultaneously. We denote by $\chi_A$ the characteristic function of a set $A$.  For simplicity reasons, we adopt the notation $y:=(y_1,y_2)\in \R^{2}$ and $\xi:=(\xi_1,\xi_2)\in \R^{2}.$

\subsection*{Acknowledgements}
The first author was supported by GACR P201/24-10505S. The second and third author were supported by the Primus research programme\\ PRIMUS/21/SCI/002 of Charles University. Furthermore, the second author gratefully acknowledges support from the Foundation for Education and European Culture, founded by Nicos and Lydia Tricha. The third author was also supported by Charles University Research Centre program No.\ UNCE/24/SCI/005.

\section{Proof of Theorem~\ref{thm:main result}: Overview}\label{S:overview}

Following the approach from~\cite{GHH18}, our proof of Theorem~\ref{thm:main result} utilizes a dyadic decomposition of Duoandikoetxea and Rubio de Francia \cite{DuRu1986}, reducing the study of an operator with a rough kernel to a series of operators with smooth kernels. To be more specific, we choose a Schwartz function $\Phi$ on $\R^2$ whose Fourier transform $\wh{\Phi}$ is supported in the annulus $\{\xi\in \R^2: 1/2\le |\xi|\le 2\}$ and satisfies the property $\sum_{j\in\Z}\wh{\Phi_j}(\xi)=1$ for $\xi \neq 0$, where $\Phi_j:=2^{2j} \Phi(2^j \cdot)$. Let 
\begin{equation*}
K(y_1,y_2)=\frac{\Omega((y_1,y_2)/|(y_1,y_2)|)}{|(y_1,y_2)|^{2}}, \quad (y_1,y_2)\in\R^{2}\setminus\{(0,0)\}
\end{equation*} 
be the kernel of the operator $T_\Omega$ defined in~\eqref{eq:T}.
We set
$$ K^{i}(y):=\wh{\Phi}(2^{i}y)K(y), \qquad i\in \Z$$
and
\begin{equation}\label{E:kij}
K_{j}^{i}(y):=\Phi_{i+j}\ast K^{i}(y), \qquad i\in \Z, ~j\in \Z.
\end{equation}
We then decompose the kernel $K$ as $K=\sum_{j\in \Z} K_j$, where 
$$K_{j}(y):=\sum_{i=-\infty}^{\infty}{K_{j}^{i}(y)},\qquad j\in\Z.$$
We denote by $T_j$ and $T^i_j$ the bilinear singular integral operators defined as in~\eqref{eq:T}, save for the fact that the kernel $K$ is replaced by $K_j$ or $K^i_j$, respectively. 

Having at our disposal the aforementioned decomposition of $T_{\Omega}$ into the series of operators $T_j$, we proceed by proving weak-type and strong-type bounds for the operators $T_j$ in the case when $j\ge j_0$ for a suitable positive integer $j_0$. We do not provide new estimates for $T_j$ in the case when $j<j_0$ since bounds for these operators with a geometric decay in $j$ were obtained in~\cite[Proposition 3]{GHH18} by making use of the multilinear Calder\'on-Zygmund theory~\cite{CM75,GT}.

The following proposition follows from~\cite[Claim 7]{DS23} via the embedding $L^p(\mathbb R) \hookrightarrow L^{p,\infty}(\mathbb R)$. In what follows, $j_0$ stands for a positive integer chosen as in~\cite[Claim 7]{DS23}.

\begin{customprop}{A}[\cite{DS23}, Claim 7]\label{pr:1}
Let $j\geq j_0$ and $1<p,p_1,p_2<\infty$ satisfy $1/p=1/p_1+1/p_2$.  Suppose that $q>1$ and $\Omega \in L^q(\S^{1})$ fulfills $\int_{\S^{1}} \Omega(\theta) d\sigma(\theta)=0$. Then there exists a constant $\delta=\delta(p_1,p_2,q)>0$  such that the estimate
\begin{equation*}
\big\Vert T_{j}(f_1,f_2)\big\Vert_{L^{p,\infty}(\R)}\lesssim 2^{-j\delta}\Vert\Omega\Vert_{L^q(\mathbb{S}^{1})}\Vert f_1\Vert_{L^{p_1}(\R)}\Vert f_2\Vert_{L^{p_2}(\R)}    
\end{equation*}
holds, up to a multiplicative constant depending on $p_1, p_2$ and $q$. 
\end{customprop}

Proposition~\ref{pr:1} is only valid for Banach tuples of exponents $(p_1,p_2,p)$. The key new ingredient needed for establishing the quasi-Banach estimates of Theorem~\ref{thm:main result} is the following $L^1(\R) \times L^1(\R) \rightarrow L^{\frac{1}{2},\infty}(\R)$ bound for the operators $T_j$, with an arbitrarily slow blowup in $j$.

\begin{prop}\label{pr:2} 
Let $j\geq j_0$, $\delta>0$, and suppose that $\Omega \in L^1(\S^{1},w)$, where $w(\theta_1,\theta_2)=\frac{1}{|\theta_1-\theta_2|}$ for $(\theta_1,\theta_2)\in \S^1$. Assume that $\int_{\S^{1}} \Omega(\theta) d\sigma(\theta)=0$. 
Then the estimate
\begin{equation*}
\big\Vert T_{j}(f_1,f_2)\big\Vert_{L^{\frac{1}{2},\infty}(\R)}\lesssim 2^{j\delta}\Vert\Omega\Vert_{L^1(\mathbb{S}^{1},w)}\Vert f_1\Vert_{L^{1}(\R)}\Vert f_2\Vert_{L^{1}(\R)}
\end{equation*}
holds, up to a multiplicative constant depending on $\delta$. 
\end{prop}

Proposition~\ref{pr:2} is proved in Section \ref{S:CZ decomposition} via geometric arguments.
In Section \ref{S:interpolation}, we then interpolate between the estimates of Proposition~\ref{pr:1} and Proposition~\ref{pr:2} by making use of the bilinear version of the Marcinkiewicz interpolation theorem (see \cite[Theorem 1.1]{GLLZ12} or \cite[Theorem 3]{Ja1988}). This yields a new strong $L^{p_1}(\R) \times L^{p_2}(\R) \rightarrow L^p(\R)$ estimate for the operator $T_j$ in the full quasi-Banach range of exponents~\eqref{E:range-full}. This estimate is formulated in Proposition \ref{pr:3} below. Thanks to the aforementioned result \cite[Proposition 3]{GHH18} dealing with operators $T_j$ with negative values of $j$, the statement of Proposition~\ref{pr:3} can be transferred into an analogous statement for the operator $T_\Omega$, as expressed in Proposition \ref{pr:4}. The proof of the first part of Theorem~\ref{thm:main result} is then concluded by performing a yet another interpolation between Proposition \ref{pr:4} and \cite[Theorem 1]{DS23}. For this final interpolation argument, we view the operator $T_\Omega(f_1,f_2)$ as a trilinear operator in $f_1$, $f_2$ and $\Omega$ and we employ the following weighted multilinear interpolation result of Cao, Olivo and Yabuta~\cite{COY22}.

\begin{customprop}{B}[\cite{COY22}, Theorem 3.1]\label{thm:COY}
Suppose that $(\Sigma_0, \mu_0), \ldots, (\Sigma_m, \mu_m)$ are
measure spaces, and $\mathscr{S}_j$ is the collection of all simple
functions
on $\Sigma_j$, $j=1,\dots,m$. Denote by $\mathfrak{M}(\Sigma_0)$ the set
of
all measurable functions on $\Sigma_0$.
Let $T: \mathscr{S}=\mathscr{S}_1 \times \cdots \times \mathscr{S}_m \to
\mathfrak{M}(\Sigma_0)$ be an $m$-linear operator.
Let $0<p_0, q_0< \infty$, $1\le p_j, q_j\le \infty$ $(j=1,\dots,m)$, and
let
$w_j, v_j$ be weights on $\Sigma_j$ $(j=0,\dots,m)$.
Assume that there exist $M_1, M_2 \in (0, \infty)$ such that
\begin{align}\label{E:1}
&\|T\|_{L^{p_1}(\Sigma_1,\, w_1^{p_1}) \times \cdots
\times L^{p_m}(\Sigma_m,\, w_m^{p_m}) \to L^{p_0}(\Sigma_0,\, w_0^{p_0})}
\le M_1,
\\
&\|T\|_{L^{q_1}(\Sigma_1,\, v_1^{q_1}) \times \cdots \times
L^{q_m}(\Sigma_m,\, v_m^{q_m}) \to L^{q_0}(\Sigma_0,\, v_0^{q_0})} \le
M_2. \label{E:2}
\end{align}
Then, we have
\begin{equation*}
\|T\|_{L^{r_1}(\Sigma_1,\, u_1^{r_1}) \times \cdots \times
L^{r_m}(\Sigma_m,\, u_m^{r_m}) \to L^{r_0}(\Sigma_0,\, u_0^{r_0})}
\le M_1^{1-\theta} M_2^{\theta},
\end{equation*}
for all exponents satisfying
\begin{equation*}
0<\theta<1,\quad \frac{1}{r_j}=\frac{1-\theta}{p_j}+\frac{\theta}{q_j}
\quad\text{and}\quad u_j=w_j^{1-\theta} v_j^{\theta},\quad j=0,\dots,m.
\end{equation*}
\end{customprop}

We are not aware of any weighted multilinear interpolation results in which the strong-type estimates~\eqref{E:1} and~\eqref{E:2} are replaced by their weak-type counterparts. For this reason, we have to utilize the two-step procedure described above, whose first step aims at passing from weak-type to strong-type bounds. 

Section \ref{S:counterexamle n=1} is dedicated to the proof of the second part of Theorem \ref{thm:main result}. We employ counterexamples similar to those in~\cite[proof of Proposition 3]{DS23}, which were in turn built upon earlier examples from~\cite{DGHST11} and~\cite{GHHP22}.

\section{Proof of Theorem~\ref{thm:main result}: Calder\'on-Zygmund decomposition}\label{S:CZ decomposition}

In this section, we prove Proposition~\ref{pr:2}. Our proof uses the Calder\'on-Zygmund decomposition, applying the results of~\cite{DS23} to estimate the ``good'' part of the operator $T_j$ and the following geometric lemma to deal with the ``bad'' part of the operator. The main geometric idea employed in the proof of Lemma~\ref{L:integral-estimate} was depicted in Figure~\ref{figure}.

\begin{lemma}\label{L:integral-estimate}
Let $i, \ell\in \Z$ and let $\Omega \in L^1(\S^1,w)$, where $w(\theta_1,\theta_2)=\frac{1}{|\theta_1-\theta_2|}$ for $(\theta_1,\theta_2)\in \S^1$. Then for any $(y_1,y_2)\in \R^2$ we have
\begin{align}\label{E:int-estimate}
\int_{\frac{1}{2}\leq |(z_1,z_2)| \leq 2} \int_{\R} 2^{i+2\ell} &\chi_{\{x:~|(x-y_1-2^{-i}z_1, x-y_2-2^{-i}z_2)| \leq 2^{-i-\ell}\}}(x) \Big|\Omega\Big(\frac{(z_1,z_2)}{|(z_1,z_2)|}\Big)\Big| \,dx dz_1 dz_2\\
&\leq 16 \|\Omega\|_{L^1(\S^1,w)}.
\nonumber
\end{align}
\end{lemma}

\begin{proof}
Assume that $(z_1,z_2)\in \R^2$ is such that the set of those $x\in \R$ satisfying 
\begin{equation}\label{E:x}
|(x-y_1-2^{-i}z_1, x-y_2-2^{-i}z_2)| \leq 2^{-i-\ell}
\end{equation}
is nonempty. Taking any $x\in \mathbb R$ satisfying~\eqref{E:x}, we have
\[
|z_2-z_1-2^i(y_1-y_2)| \leq |z_2-2^ix+2^i y_2| + |2^i x-2^iy_1-z_1| \leq 2^{1-\ell}.
\]
In addition, we observe that for fixed $(y_1,y_2)$ and $(z_1,z_2)$, the set of those $x\in \mathbb R$ satisfying~\eqref{E:x} has measure at most $2^{1-i-\ell}$. Thus, the integral on the left-hand side of~\eqref{E:int-estimate} is bounded by
\[
2^{1+\ell}\int_{\{(z_1,z_2): ~\frac{1}{2}\leq |(z_1,z_2)| \leq 2, ~|z_2-z_1-2^i(y_1-y_2)| \leq 2^{1-\ell}\}} \Big|\Omega\Big(\frac{(z_1,z_2)}{|(z_1,z_2)|}\Big)\Big|\,dz_1 dz_2.
\]
This can be rewritten in polar coordinates as
\begin{equation}\label{E:polar}
2^{1+\ell}\int_{\S^1} |\Omega(\theta_1,\theta_2)| \int_{R(\theta_1,\theta_2)} r \,dr d\sigma(\theta_1,\theta_2),
\end{equation}
where 
\[
R(\theta_1,\theta_2)=\{r \in [1/2,2]: |(\theta_2-\theta_1)r -2^i(y_1-y_2)| \leq 2^{1-\ell}\}.
\]
Fix $(\theta_1,\theta_2)\in \S^1$. If $r\in R(\theta_1,\theta_2)$ then $(\theta_2-\theta_1)r$ belongs to an interval of length $2^{2-\ell}$, and thus the measure of $R(\theta_1,\theta_2)$ is at most $2^{2-\ell}/|\theta_2-\theta_1|$. Estimating~\eqref{E:polar} with the help of the previous observation, we obtain~\eqref{E:int-estimate}, as desired.
\end{proof}

\begin{proof}[Proof of Proposition~\ref{pr:2}]
Throughout the proof, we denote the Lebesgue measure of a measurable subset $A$ of $\R$ by $|A|$. For each interval $Q$ in $\R$, let $\ell(Q)$ and $c_Q$ denote the length and the center of $Q$, respectively.

By a simple scaling argument, we may assume that 

\begin{equation*}
\Vert f_1\Vert_{L^1(\R)}=\Vert f_2\Vert_{L^1(\R)}=\Vert\Omega\Vert_{L^1(\S^1,w)}=1.   
\end{equation*}
Then it suffices to show that for all $\delta>0$ and $\lambda>0$,
\begin{equation}\label{weakgoal}
\Big| \Big\{x\in\R : \big| T_{j}\big(f_1,f_2\big)(x)\big|>\lambda\Big\}\Big|\lesssim 2^{\delta j}\lambda^{-\frac{1}{2}},
\end{equation}
up to a constant depending on $\delta$.
We note that the arbitrary parameter $\delta>0$ corresponds to $\delta/2$ in the statement of the proposition.

Now, we perform a Calder\'on--Zygmund decomposition of $f_1$ and $f_2$ at height $\lambda^{\frac{1}{2}}$. More specifically, we express $f_1$ and $f_2$ as 
$$f_1=g_1+\sum_{Q\in \mathcal{A}}b_{1,Q}\qquad\text{and}\qquad f_2=g_2+\sum_{Q'\in \mathcal{B}}b_{2,Q'},$$
where 

\begin{enumerate}[label=(\roman*)]
    \item $\mathcal{A},\mathcal{B}$ are collections of disjoint dyadic intervals such that $$\bigg|\bigcup_{Q\in\mathcal{A}} Q\bigg|\leq \lambda^{-\frac{1}{2}}
    \quad \text{and} \quad \bigg|\bigcup_{Q'\in\mathcal{B}} Q'\bigg|\\\leq \lambda^{-\frac{1}{2}},$$ \label{itm:1}
    \item  $\supp(b_{1,Q})\subset Q$ and $\supp(b_{2,Q'})\subset Q',$\label{itm:2}
    \item $$\int_{\R}{b_{1,Q}(y)}dy=0 \quad \text{and} \quad \int_{\R} {b_{2,Q'}(y)}dy=0,$$\label{itm:3}
    \item $$\Vert b_{1,Q}\Vert_{L^1(\R)}\leq 4 \lambda^{\frac{1}{2}}|Q| \quad \text{and} \quad 
    \Vert b_{2,Q'}\Vert_{L^1(\R)}\leq 4  \lambda^{\frac{1}{2}}|Q'|,$$\label{itm:4}
    \item $$\Vert g_1\Vert_{L^r(\R)}\leq 2 \lambda^{(1-\frac{1}{r})\frac{1}{2}} \quad \text{and} \quad \Vert g_2\Vert_{L^r(\R)}\leq 2 \lambda^{(1-\frac{1}{r})\frac{1}{2}} \quad\text{for all}\quad 1\leq r\leq \infty.$$\label{itm:5}
\end{enumerate}
Given $Q \in \mathcal{A}$, we further define
\[
f_{2,Q}=g_2+\sum_{Q'\in \mathcal{B}:~\ell(Q') \leq \ell(Q)} b_{2,Q'},
\]
and similarly for $Q'\in \mathcal{B}$ we set
\[
f_{1,Q'}=g_1+\sum_{Q\in \mathcal{A}:~\ell(Q) < \ell(Q')} b_{1,Q}.
\]
Thus, we can estimate the left-hand side of \eqref{weakgoal} in the following way:
\begin{align}\label{weakgoal:eq 1}
    \Big| \Big\{x\in\R : \big| T_{j}\big(f_1,f_2\big)(x)\big|>\lambda\Big\}\Big|&\leq\Big| \Big\{x\in\R : \big|T_{j}\big(g_1,g_2\big)(x)\big|>\frac{\lambda}{3} \Big\}\Big|\\\nonumber
    &\qquad+\Big|\Big\{x\in\R : \sum_{Q \in \mathcal{A}} \Big| T_{j}\Big(b_{1,Q},f_{2,Q}\Big)(x)\Big|>\frac{\lambda}{3}\Big\}\Big|\\\nonumber
    &\qquad+\Big|\Big\{x\in\R : \sum_{Q' \in \mathcal{B}}  \Big| T_{j}\Big(f_{1,Q'},b_{2,Q'}\Big)(x)\Big|>\frac{\lambda}{3}\Big\}\Big|.
\end{align}

We estimate the first term on the right-hand side of \eqref{weakgoal:eq 1} as follows:
\begin{align*}
    \Big| \Big\{x\in\R : \big|T_{j}\big(g_1,g_2\big)(x)\big|>\frac{\lambda}{3} \Big\}\Big|&\leq 9\lambda^{-2}\Vert T_j(g_1,g_2)\Vert_{L^2(\R)}^2\\\nonumber
    &\lesssim 2^{\delta j}\lambda^{-2}\Vert g_1\Vert_{L^4(\R)}^2\Vert g_2\Vert_{L^4(\R)}^2\\\nonumber
    &\lesssim 2^{\delta j} \lambda^{-\frac{1}{2}},
\end{align*}
up to constants depending on $\delta$.
We note that the penultimate inequality follows from~\cite[Proposition 4]{DS23}, while the last inequality follows by property \ref{itm:5} of the good functions $g_1$ and $g_2$. 

Since the second and third term on the right-hand side of~\eqref{weakgoal:eq 1} are symmetric, we only estimate the second one. We first observe that this term is bounded by
\[
\bigg|\bigcup_{Q\in\mathcal{A}} Q^{*}\bigg| +\bigg|\Big\{x\in \Big( \bigcup_{Q\in\mathcal{A}}Q^*\Big)^c:~\sum_{Q \in \mathcal{A}} \Big|T_{j}\Big(b_{1,Q},f_{2,Q}\Big)(x)\Big|>\frac{\lambda}{3}\Big\}\bigg|.
\]
Here and in what follows, $Q^{*}$ stands for the interval with the same center as $Q$ satisfying $\ell(Q^*)=10\ell(Q)$. By property \ref{itm:1}, we have $$\bigg|\bigcup_{Q\in\mathcal{A}}Q^{*}\bigg|\leq 10 \bigg|\bigcup_{Q\in\mathcal{A}}Q\bigg|\leq 10 \lambda^{-\frac{1}{2}}.$$ 
We next set
\begin{align*}
    E:=\Big\{x\in \Big( \bigcup_{Q\in\mathcal{A}}Q^*\Big)^c:~\sum_{Q \in \mathcal{A}} \Big|T_{j}\Big(b_{1,Q},f_{2,Q}\Big)(x)\Big|>\frac{\lambda}{3}\Big\}
\end{align*}
and estimate $|E|$.

Chebyshev's inequality yields
\begin{align*}
    |E|&\leq 3\lambda^{-1} \sum_{Q\in\mathcal{A}} \sum_{i\in\Z} \int_{(\bigcup_{Q\in\mathcal{A}}Q^*)^c} \big|T_j^i\big(b_{1,Q},f_{2,Q} \big)(x)\big|dx\\\nonumber
    &=3\lambda^{-1} \sum_{Q\in\mathcal{A}}\sum_{i:2^{i}\ell(Q)> 1}\int_{(\bigcup_{Q\in\mathcal{A}}Q^*)^c}\big|T_j^i\big(b_{1,Q},f_{2,Q} \big)(x)\big|dx\\\nonumber
    &\qquad+3\lambda^{-1} \sum_{Q\in\mathcal{A}}\sum_{i:2^{i}\ell(Q)\leq 1}\int_{(\bigcup_{Q\in\mathcal{A}}Q^*)^c}\big|T_j^i\big(b_{1,Q},f_{2,Q} \big)(x)\big|dx,\\\nonumber
    &=:\mathcal{I}_1+\mathcal{I}_2,
\end{align*}
where $T_j^i$ is the bilinear operator associated with the kernel $K^i_j$ introduced in~\eqref{E:kij}.

We dominate the Schwartz function $\Phi$ as
\begin{equation}\label{E:est-phi}
|\Phi(a,b)|\lesssim \sum_{k=1}^{\infty}2^{-5k}\chi_{B(0,2^k)}(a,b),    
\end{equation}
where $B(0,2^k)$ stands for the two-dimensional ball centered at the origin and with radius $2^k$, and the inequality holds up to a constant depending only on  $\Phi$. Since $\Phi$ is fixed throughout the paper, the constant in~\eqref{E:est-phi} is in fact an absolute constant. 
Throughout the rest of this proof and unless stated otherwise, the constants in the relation ``$\lesssim$'' are supposed to be absolute as well.

We then estimate the term $\mathcal{I}_1$ as
\begin{align}\label{E:I1}
    \mathcal{I}_1&\lesssim \lambda^{-1} \sum_{Q\in\mathcal{A}}\sum_{i:2^{i}\ell(Q)>1} \sum_{k=1}^\infty 2^{-5k} \int_{\R^2} \int_{\frac{1}{2}\leq |(z_1,z_2)|\leq 2} \int_{(\bigcup_{Q\in\mathcal{A}}Q^*)^c} \Big|\Omega\Big(\frac{(z_1,z_2)}{|(z_1,z_2)|}\Big)\Big|2^{2i+2j}\\\nonumber
    &\quad\times\chi_{\{|(x-y_1-2^{-i}z_1,x-y_2-2^{-i}z_2)|\leq 2^{-i-j+k}\}}(x)
    |b_{1,Q}(y_1)| |f_{2,Q}(y_2)| dx dz_1 dz_2 d y_1 d y_2.
\end{align}
Assume that $x\in (\bigcup_{Q\in\mathcal{A}}Q^*)^c$ and $(y_1,y_2) \in Q \times \R$ satisfy $|(x-y_1-2^{-i}z_1,x-y_2-2^{-i}z_2)|\leq 2^{-i-j+k}$ for some $(z_1,z_2)$ with $1/2 \leq |(z_1,z_2)| \leq 2$. Since $2^i \ell(Q) >1$, we obtain
\[
\ell(Q) \leq |x-y_1-2^{-i}z_1| \leq 2^{-i-j+k} \leq 2^{-i+k},
\]
and thus the sum in $k$ on the right-hand side of~\eqref{E:I1} can be taken only over those $k$ for which $2^k \geq \ell(Q)2^{i}$. In addition,
\begin{align*}
|y_1-y_2| &\leq |y_1-x+2^{-i} z_1|+|-2^{-i}z_1+2^{-i}z_2|+|2^{-i}z_2-x+y_2|\\
&\leq 2^{1-i-j+k} +2^{2-i} \leq 2^{k+3} \ell(Q),
\end{align*}
which yields that $y_2$ has to belong to the interval $I(c_Q,2^{k+4} \ell(Q))$ centered at $c_Q$ and with radius $2^{k+4} \ell(Q)$.
Using Lemma~\ref{L:integral-estimate} with $\ell=j-k$, we then bound the right-hand side of~\eqref{E:I1} by a multiple of
\begin{align}\label{E:I1-1}
    &\lambda^{-1} \sum_{Q\in\mathcal{A}}\sum_{i:2^{i}\ell(Q)>1} \sum_{k:2^k \geq \ell(Q) 2^{i}} 2^{i-3k} \int_{\R}
    |b_{1,Q}(y_1)| dy_1 \int_{I(c_Q,2^{k+4} \ell(Q))} |f_{2,Q}(y_2)| d y_2.
\end{align}
Using property \ref{itm:5} of the function $g_2$, property~\ref{itm:4} of the functions $b_{2,Q'}$, inequality $\ell(Q') \leq \ell(Q)$ and the fact that the supports of the functions $b_{2,Q'}$ are pairwise disjoint, we obtain
\[
\int_{I(c_Q,2^{k+4} \ell(Q))} |f_{2,Q}(y_2)| d y_2 \leq \lambda^{\frac{1}{2}} 2^{k+10} \ell(Q).
\]
Therefore, \eqref{E:I1-1} is bounded by
\begin{align*}
    &\lambda^{-\frac{1}{2}} \sum_{Q\in\mathcal{A}}\sum_{i:2^{i}\ell(Q)>1} \sum_{k:2^k \geq \ell(Q) 2^{i}} 2^{i}\ell(Q) 2^{-2k} \int_{\R}
    |b_{1,Q}(y_1)| dy_1\\
    &\lesssim \lambda^{-\frac{1}{2}} \sum_{Q\in\mathcal{A}} \int_{\R}
    |b_{1,Q}(y_1)| dy_1 \sum_{i:2^{i}\ell(Q)>1} 
    (2^{i} \ell(Q))^{-1} \\
    &\lesssim \lambda^{-\frac{1}{2}}.
\end{align*}
In the above chain, we used  property~\ref{itm:4} of the functions $b_{1,Q}$ and property~\ref{itm:1} of the collection $\mathcal{A}$.

It remains to estimate the term $\mathcal{I}_2$. Let $\delta>0$ be as in~\eqref{weakgoal}. In fact, since estimate~\eqref{weakgoal} becomes weaker as $\delta$ increases, we may assume that $\delta \in (0,1/2)$. As $\Phi$ is a Schwartz function, we obtain
\begin{align*}
|\Phi(a,b)-\Phi(\widetilde{a},b)| &\lesssim \min\{|a-\widetilde{a}|, (1+|(a,b)|)^{-\frac{4}{1-\delta}}+(1+|(\widetilde{a},b)|)^{-\frac{4}{1-\delta}}\}\\
&\lesssim |a-\widetilde{a}|^\delta \big((1+|(a,b)|)^{-4}+(1+|(\widetilde{a},b)|)^{-4}\big)\\
&\lesssim |a-\widetilde{a}|^\delta \sum_{k=1}^{\infty}2^{-4k}(\chi_{B(0,2^k)}(a,b)+\chi_{B(0,2^k)}(\widetilde{a},b)),
\end{align*}
for $a$, $\widetilde{a}$, $b\in \R$. Using this estimate and the vanishing moment property \ref{itm:3} of the bad functions $b_{1,Q}$, we deduce that
\begin{align}\label{E:I2}
    \mathcal{I}_2&\lesssim    
    \lambda^{-1} \sum_{Q\in\mathcal{A}}{\sum_{i:2^{i}\ell(Q)\leq 1}} \int_{\R^2} \int_{\frac{1}{2}\leq |(z_1,z_2)|\leq  2}\int_{(\bigcup_{Q\in\mathcal{A}}Q^*)^c}\Big|\Omega\Big(\frac{(z_1,z_2)}{|(z_1,z_2)|}\Big)\Big|\\\nonumber
    &\;\;\times\big|\Phi_{i+j}(x-y_1-2^{-i}z_1,x-y_2-2^{-i}z_2)-\Phi_{i+j}(x-c_{Q}-2^{-i}z_1,x-y_2-2^{-i}z_2)\big|\\\nonumber
    &\;\;\times|b_{1,Q}(y_1)| |f_{2,Q}(y_2)|dx d z_1 d z_2 d y_1 d y_2\\\nonumber
&\lesssim \lambda^{-1} \sum_{Q\in\mathcal{A}}\sum_{i:2^{i}\ell(Q) \leq 1} \sum_{k=1}^\infty 2^{-4k} \int_{\R^2} \int_{\frac{1}{2}\leq |(z_1,z_2)|\leq 2} \int_{(\bigcup_{Q\in\mathcal{A}}Q^*)^c} \Big|\Omega\Big(\frac{(z_1,z_2)}{ |(z_1,z_2)|}\Big)\Big|2^{2i+2j}\\\nonumber
    &\;\;\times (\chi_{\{|(x-y_1-2^{-i}z_1,x-y_2-2^{-i}z_2)|\leq 2^{-i-j+k}\}}(x)+\chi_{\{|(x-c_Q-2^{-i}z_1,x-y_2-2^{-i}z_2)|\leq 2^{-i-j+k}\}}(x))\\
    &\;\;\times (2^{i+j}|y_1-c_Q|)^\delta
    |b_{1,Q}(y_1)| |f_{2,Q}(y_2)| dx dz_1 dz_2 d y_1 d y_2.
    \nonumber
\end{align}
The right-hand side of the previous display naturally splits into two terms, each involving one characteristic function in the variable $x$. We only estimate the term involving $\chi_{\{|(x-y_1-2^{-i}z_1,x-y_2-2^{-i}z_2)|\leq 2^{-i-j+k}\}}(x)$ as the other one can be treated similarly. We observe that the condition $|(x-y_1-2^{-i}z_1,x-y_2-2^{-i}z_2)|\leq 2^{-i-j+k}$ paired with $\frac{1}{2}\leq |(z_1,z_2)| \leq 2$ yields
\[
|y_1-y_2| \leq |y_1-x+2^{-i} z_1|+|-2^{-i}z_1+2^{-i}z_2|+|2^{-i}z_2-x+y_2|\leq 2^{1-i-j+k} +2^{2-i} \leq2^{2-i+k}.
\]
Using also the fact that $y_1 \in Q$ and $\ell(Q) \leq 2^{-i}$, we deduce that 
$y_2$ has to belong to the interval $I(c_Q,2^{3-i+k})$ centered at $c_Q$ and with radius $2^{3-i+k}$. Applying this, \eqref{E:I2} and Lemma~\ref{L:integral-estimate} with $\ell=j-k$, we obtain
\begin{align}\label{E:I2-2}
\mathcal{I}_2&\lesssim    
\lambda^{-1} \sum_{Q\in\mathcal{A}}\sum_{i:2^{i}\ell(Q) \leq 1} \sum_{k=1}^\infty 2^{i-2k} (2^{i+j}\ell(Q))^\delta\\ 
&\quad \times \int_{\R}
    |b_{1,Q}(y_1)| dy_1 \int_{I(c_Q,2^{3-i+k})} |f_{2,Q}(y_2)| d y_2.
    \nonumber
\end{align}
Using property \ref{itm:5} of the function $g_2$, property~\ref{itm:4} of the functions $b_{2,Q'}$, inequality $\ell(Q') \leq \ell(Q) \leq 2^{-i}$ and the fact that the supports of the functions $b_{2,Q'}$ are pairwise disjoint, we obtain
\begin{equation}\label{E:I2-3}
\int_{I(c_Q,2^{3-i+k})} |f_{2,Q}(y_2)| d y_2 \leq\lambda^{\frac{1}{2}} 2^{10-i+k}.
\end{equation}
A combination of~\eqref{E:I2-2} and \eqref{E:I2-3} yields
\begin{align*}
\mathcal{I}_2 &\lesssim
\lambda^{-\frac{1}{2}} 2^{j\delta} \sum_{Q\in\mathcal{A}}\sum_{i:2^{i}\ell(Q) \leq 1} \sum_{k=1}^\infty 2^{-k} (2^{i}\ell(Q))^\delta
\int_{\R} |b_{1,Q}(y_1)| dy_1\\
&\lesssim \lambda^{-\frac{1}{2}} 2^{j\delta},
\end{align*}
where the constant in the last inequality depends on $\delta$. We note that this inequality follows by summing geometric series in $k$ and $i$ and by using property~\ref{itm:4} of the functions $b_{1,Q}$ and property~\ref{itm:1} of the collection $\mathcal{A}$. This completes the proof.
\end{proof}

\section{Proof of Theorem~\ref{thm:main result}: Interpolation}\label{S:interpolation}

In this section, we finish the proof of Theorem~\ref{thm:main result} by performing various interpolation arguments. We first interpolate between the results of Proposition~\ref{pr:1} and Proposition~\ref{pr:2}, establishing the $L^{p_1}(\R) \times L^{p_2}(\R) \rightarrow L^p(\R)$ boundedness of the family of operators $T_j$ with a geometric decay in $j$ under somewhat stronger assumptions on $\Omega$ than those stated in Theorem~\ref{thm:main result}. 

\begin{prop}\label{pr:3}
Let $j\geq j_0$ and $1<p_1,p_2<\infty$, $1/2<p <1$ satisfy $1/p=1/p_1+1/p_2$. Suppose that $\Omega \in L^q(\S^{1},\widetilde{w}^q)$, where $q>1$, $\widetilde{w}(\theta_1,\theta_2)=\frac{1}{|\theta_1-\theta_2|^\alpha}$ for $\alpha>\frac{1}{q}$ and $(\theta_1,\theta_2)\in \S^1$. Assume that $\int_{\S^{1}} \Omega(\theta) d\sigma(\theta)=0$. Then we have 

\begin{equation}\label{main ineq 4}
\big\Vert T_{j}(f_1,f_2)\big\Vert_{L^{p}(\R)}\lesssim
2^{-j\delta}\Vert\Omega\Vert_{L^q(\S^{1},\widetilde{w}^q)}\Vert f_1\Vert_{L^{p_1}(\R)}\Vert f_2\Vert_{L^{p_2}(\R)} 
\end{equation}
for some $\delta=\delta(p_1,p_2,q)>0$, up to a multiplicative constant depending on $p_1$, $p_2$, $q$ and $\alpha$.  
\end{prop}

\begin{proof}
Let $q>1$ and let $w$ be the weight introduced in the statement of Proposition~\ref{pr:2}. We start by establishing the embeddings $L^q(\S^{1},\widetilde{w}^q)\hookrightarrow L^{q}(\S^{1})$ and $L^q(\S^{1},\widetilde{w}^q)\hookrightarrow L^1(\S^{1},w)$. The first embedding follows directly from the fact that $\widetilde{w}^q$ is bounded from below by a positive constant, while the second one is proved using H{\"o}lder's inequality as follows:
\begin{align*}
\Vert\Omega\Vert_{L^1(\S^{1},w)}&=\int_{\S^1}|\Omega(\theta_1,\theta_2)|\frac{1}{|\theta_1-\theta_2|^{\alpha}}\frac{1}{|\theta_1-\theta_2|^{1-\alpha}}\;d\sigma(\theta_1,\theta_2)\\
&\leq\Vert\Omega\Vert_{L^q(\S^{1},\widetilde{w}^q)}\bigg(\int_{\S^1}\frac{1}{|\theta_1-\theta_2|^{(1-\alpha)q'}}\;d\sigma(\theta_1,\theta_2)\bigg)^{\frac{1}{q'}},
\end{align*}
where the last integral in the previous inequality is finite since $\alpha>\frac{1}{q}$.

Let $\varepsilon \in (0,1/2)$. Thanks to the aforementioned embeddings, it follows from Proposition ~\ref{pr:1} and  Proposition ~\ref{pr:2} that for some $\delta_0=\delta_0(\varepsilon,q)>0$ and for arbitrary $\delta_1>0$,
\begin{align*}
\big\Vert T_{j}\big\Vert_{L^1(\R)\times L^1(\R)\to L^{\frac{1}{2},\infty}(\R)}&\lesssim 2^{j\delta_1}\Vert\Omega\Vert_{L^q(\S^{1},\widetilde{w}^q)},\\
\big\Vert T_{j}\big\Vert_{L^\frac{1}{1-2\varepsilon}(\R)\times L^\frac{1}{\varepsilon}(\R)\to L^{\frac{1}{1-\varepsilon},\infty}(\R)}&\lesssim 2^{-j\delta_0}\Vert\Omega\Vert_{L^q(\S^{1},\widetilde{w}^q)},\\
\big\Vert T_{j}\big\Vert_{L^\frac{1}{\varepsilon}(\R)\times L^\frac{1}{1-2\varepsilon}(\R)\to L^{\frac{1}{1-\varepsilon},\infty}(\R)}&\lesssim 2^{-j\delta_0}\Vert\Omega\Vert_{L^q(\S^{1},\widetilde{w}^q)},
\end{align*}
up to constants depending on $\varepsilon$, $q$, $\delta_1$ and $\alpha$.

Now, we fix $1<p_1,p_2<\infty$ and $p<1$ satisfying $1/p=1/p_1+1/p_2$. By choosing $\varepsilon>0$ small enough depending on $p_1$, $p_2$, we may ensure that the point $(1/p_1,1/p_2)$ lies in the interior of the convex hull of the points $(1,1)$, $(1-2\varepsilon,\varepsilon)$ and $(\varepsilon, 1-2\varepsilon)$. Therefore, there exist  $0<\theta_0,\theta_1,\theta_2<1$ with 
\begin{equation*}
\theta_0+\theta_1+\theta_2=1
\end{equation*}
such that 
\begin{align}
\frac{1}{p_1}&=\theta_0+(1-2\varepsilon)\theta_1+\varepsilon\theta_2,\label{main ineq 6}\\
\frac{1}{p_2}&=\theta_0+\varepsilon\theta_1+(1-2\varepsilon)\theta_2\label{main ineq 7}.
\end{align}
Adding equations~\eqref{main ineq 6} and~\eqref{main ineq 7}, we also obtain
\[
\frac{1}{p}=2\theta_0+(1-\varepsilon)\theta_1+(1-\varepsilon)\theta_2.
\]
An application of the bilinear version of the Marcinkiewicz interpolation theorem (see \cite[Theorem 1.1]{GLLZ12} or \cite[Theorem 3]{Ja1988}) yields that
\[
\big\Vert T_{j}\big\Vert_{L^{p_1}(\R)\times L^{p_2}(\R)\to L^{p}(\R)}\lesssim 
2^{j\delta_1 \theta_0-j\delta_0 \theta_1-j\delta_0 \theta_2}\Vert\Omega\Vert_{L^q(\S^{1},\widetilde{w}^q)},
\]
up to a constant depending on $p_1$, $p_2$, $q$, $\delta_1$ and $\alpha$.
Choosing $0<\delta_1<\delta_0\frac{\theta_1+\theta_2}{\theta_0}$ and setting $\delta=-\delta_1 \theta_0+\delta_0 \theta_1+\delta_0 \theta_2>0$, we obtain the desired inequality~\eqref{main ineq 4}.
\end{proof}

By combining Proposition \ref{pr:3} and \cite[Proposition 3]{GHH18}, we establish the following boundedness result for the bilinear operator $T_{\Omega}$.

\begin{prop}\label{pr:4}
Let $1<p_1,p_2<\infty$, $1/2<p <1$ satisfy $1/p=1/p_1+1/p_2$ and suppose that $\Omega \in L^q(\S^{1},\widetilde{w}^q)$, where $q>1$, $\widetilde{w}(\theta_1,\theta_2)=\frac{1}{|\theta_1-\theta_2|^\alpha}$ for $\alpha>\frac{1}{q}$ and $(\theta_1,\theta_2)\in \S^1$. Assume that $\int_{\S^{1}} \Omega(\theta) d\sigma(\theta)=0$. Then we have 

\begin{equation}\label{main ineq 17}
\big\Vert T_{\Omega}(f_1,f_2)\big\Vert_{L^{p}(\R)}\lesssim
\Vert\Omega\Vert_{L^q(\S^{1},\widetilde{w}^q)}\Vert f_1\Vert_{L^{p_1}(\R)}\Vert f_2\Vert_{L^{p_2}(\R)},
\end{equation}
 up to a multiplicative constant depending on $p_1$, $p_2$, $q$ and $\alpha$. 
\end{prop}

\begin{proof}
We recall that $T_{\Omega}$ can be decomposed as
\begin{equation*}
T_{\Omega}=\sum_{j=-\infty}^{\infty}T_j.
\end{equation*}
Then 
\begin{equation}\label{main ineq Tomega}
\big\Vert T_{\Omega}(f_1,f_2)\big\Vert_{L^p(\R)}\le
\Big(\sum_{j\in\Z}\big\Vert T_{j}(f_1,f_2)\big\Vert_{L^p(\R)}^{p}\Big)^{\frac{1}{p}},\quad\text{where}\quad 0<p<1.
\end{equation}
It follows from \cite[Proposition 3]{GHH18} that 
\begin{equation*}
\big\Vert T_{j}(f_1,f_2)\big\Vert_{L^p(\R)}\lesssim 2^{j/2} \Vert\Omega\Vert_{L^{q}(\S^1)}\Vert
f_1\Vert_{L^{p_1}(\R)}\Vert f_2\Vert_{L^{p_2}(\R)},
\end{equation*}
up to a constant depending on $p_1$, $p_2$ and $q$. 
Consequently,
\begin{equation*}
\Big( \sum_{j<j_0} \big\Vert
T_{j}(f_1,f_2)\big\Vert_{L^p(\R)}^{p}\Big)^{\frac{1}{p}}\lesssim \Vert \Omega\Vert_{L^q(\mathbb{S}^{1})}\Vert
f_1\Vert_{L^{p_1}(\R)}\Vert f_2\Vert_{L^{p_2}(\R)}.    
\end{equation*}
Now, the embedding $L^q(\S^{1},\widetilde{w}^q)\hookrightarrow L^{q}(\S^{1})$, where $\widetilde{w}(\theta_1,\theta_2)=\frac{1}{|\theta_1-\theta_2|^{\alpha}}$ for $\alpha>\frac{1}{q}$, implies that 
\begin{equation}\label{main ineq Tj<j0}
\Big( \sum_{j<j_0} \big\Vert
T_{j}(f_1,f_2)\big\Vert_{L^p(\R)}^{p}\Big)^{\frac{1}{p}}\lesssim \Vert\Omega\Vert_{L^{q}(\S^1,\widetilde{w}^q)}\Vert
f_1\Vert_{L^{p_1}(\R)}\Vert f_2\Vert_{L^{p_2}(\R)}.
\end{equation}
In~\eqref{main ineq Tj<j0} as well as in the remaining part of this proof, the constant in the relation $\lesssim$ is allowed to depend on $p_1$, $p_2$, $q$ and $\alpha$. 
By recalling Proposition \ref{pr:3}, we know that for $j\ge j_0$ there exists a small constant $\delta_0=\delta_0(p_1,p_2,q)>0$, such that

\begin{equation*}
\big\Vert T_{j}(f_1,f_2)\big\Vert_{L^p(\R)}\lesssim 2^{-j\delta_0} \Vert\Omega\Vert_{L^{q}(\S^1,\widetilde{w}^q)}\Vert
f_1\Vert_{L^{p_1}(\R)}\Vert f_2\Vert_{L^{p_2}(\R)},
\end{equation*}
which implies that
\begin{equation}\label{main ineq Tj>j0}
\Big(\sum_{j\ge j_0} \big\Vert
T_{j}(f_1,f_2)\big\Vert_{L^p(\R)}^{p}\Big)^{\frac{1}{p}}\lesssim \Vert\Omega\Vert_{L^{q}(\S^1,\widetilde{w}^q)}\Vert f_1\Vert_{L^{p_1}(\R)}\Vert f_2\Vert_{L^{p_2}(\R)}.    
\end{equation}
Hence, by combining \eqref{main ineq Tomega}, \eqref{main ineq Tj<j0} and \eqref{main ineq Tj>j0} we conclude the desired estimate \eqref{main ineq 17}.
\end{proof}

In order to finish the proof of Theorem~\ref{thm:main result}, we will interpolate further between 
Proposition~\ref{pr:4} and \cite[Theorem 1]{DS23} via Proposition \ref{thm:COY}.

\begin{proof}[Proof of Theorem~\ref{thm:main result}, part \textup{(i)}]
In order to proceed with our interpolation argument, we view $T_\Omega$ as a trilinear operator. More specifically, let 

\begin{equation*}
\widetilde{T}(\Omega, f_1,f_2)(x):= p.v. \int_{\R^{2}} \left(\Omega - 2Avg(\Omega) \chi_{\{(\theta_1,\theta_2): \theta_1 \theta_2 <0\}}\right)(y')|y|^{-2} f_1(x-y_1)f_2(x-y_2) dy,   
\end{equation*}
where $Avg(\Omega) =\frac{\int_{\S^{1}} \Omega d\s}{2\pi}$ is the mean value of $\Omega$. Notice that $\widetilde{T}$ is trilinear and $\widetilde{T}(\Omega,f_1,f_2)$ agrees with $T_\Omega(f_1,f_2)$ for every $\Omega$ with mean value zero and for any Schwartz functions $f_1,f_2$. 

Let $\varepsilon\in(0,1/2)$, $\beta\in(\varepsilon,1-\varepsilon)$ and $\widetilde{q}\in(1,\infty)$. It follows from Proposition ~\ref{pr:4} and \cite[Theorem 1]{DS23} that 
\begin{align}
&\big\Vert \widetilde{T}\big\Vert_{L^{\tilde{q}}(\S^{1},\widetilde{v}^{\tilde{q}}) \times L^\frac{1}{1-\varepsilon}(\R)\times L^\frac{1}{1-\varepsilon}(\R)\to L^{\frac{1}{2-2\varepsilon}}(\R)}\lesssim 1\label{main ineq 18},\\
&\big\Vert \widetilde{T}\big\Vert_{L^{1+\varepsilon}(\S^{1}) \times L^\frac{1}{\beta-\varepsilon}(\R)\times L^\frac{1}{1-\beta-\varepsilon}(\R)\to L^{\frac{1}{1-2\varepsilon}}(\R)}\lesssim 1\label{main ineq 19}.
\end{align}
Here, $\widetilde{v}(\theta_1,\theta_2)=\frac{1}{|\theta_1-\theta_2|^{\widetilde{\alpha}}}$ for $\widetilde{\alpha}>1/\widetilde{q}$ and $(\theta_1,\theta_2)\in \S^1$, and the above inequalities hold
up to constants depending on $\varepsilon$, $\beta$, $\widetilde{q}$ and $\widetilde{\alpha}$.
Now, we fix $1<p_1,p_2<\infty$, $p$ satisfying $1/p=1/p_1+1/p_2$, and suppose $q>1$ is such that $1/p+1/q\geq 2$. Then for any $\varepsilon>0$ small enough, we choose the parameters $\beta$ and $\widetilde{q}$ in inequalities~\eqref{main ineq 18} and~\eqref{main ineq 19} in such a way that the point $(1/q, 1/p_1,1/p_2)$ lies inside the line segment in $\R^3$ connecting the points $(1/\widetilde{q},1-\varepsilon, 1-\varepsilon)$ and $(1/(1+\varepsilon), \beta-\varepsilon, 1-\beta-\varepsilon)$. More specifically, the following equations are satisfied for some $0<\theta<1$:
\begin{align}
\frac{1}{p_1}&=(1-\varepsilon)(1-\theta)+(\beta-\varepsilon)\theta,\label{main ineq 20}\\
\frac{1}{p_2}&=(1-\varepsilon)(1-\theta)+(1-\beta-\varepsilon)\theta,\label{main ineq 21}\\
\frac{1}{p}&=(2-2\varepsilon)(1-\theta)+(1-2\varepsilon)\theta,\label{main ineq 22}\\
\frac{1}{q}&=\frac{1}{\tilde{q}}(1-\theta)+\frac{\theta}{1+\varepsilon}\label{main ineq 23}.
\end{align}
It is not hard to see that from \eqref{main ineq 20}-\eqref{main ineq 23} we get that 
\begin{align*}
\beta&=\frac{1}{2-2\varepsilon-\frac{1}{p}}\bigg(\varepsilon-1+\frac{1}{p_1}\bigg)+1,\\
\widetilde{q}&=\frac{(1+\varepsilon)(-1+2\varepsilon+\frac{1}{p})}{-2+\frac{1}{p}+\frac{1}{q}+2\varepsilon+\frac{\varepsilon}{q}},\\
\theta&=2-2\varepsilon-\frac{1}{p}.
\end{align*}
An application of Proposition \ref{thm:COY} implies that
\begin{equation}\label{E:widetilde-T}
\big\Vert \widetilde{T}\big\Vert_{L^q(\S^{1},u^q) \times L^{p_1}(\R)\times L^{p_2}(\R)\to L^{p}(\R)}\lesssim
1,
\end{equation}
up to constants depending on $p_1$, $p_2$, $q$, $\varepsilon$ and $\widetilde{\alpha}$. Here,
\[
u(\theta_1,\theta_2)=\widetilde{v}(\theta_1,\theta_2)^{1-\theta}=\bigg(\frac{1}{|\theta_1-\theta_2|}\bigg)^{\widetilde{\alpha}(-1+2\varepsilon+\frac{1}{p})} \quad \text{for} \quad \widetilde{\alpha}>\frac{-2+\frac{1}{p}+\frac{1}{q}+2\varepsilon+\frac{\varepsilon}{q}}{(1+\varepsilon)(-1+2\varepsilon+\frac{1}{p})}.
\]
Choosing 
\begin{equation*}
\widetilde{\alpha}=\frac{(1+\varepsilon)(-2+\frac{1}{p}+\frac{1}{q}+2\varepsilon)+\varepsilon}{(1+\varepsilon)(-1+2\varepsilon+\frac{1}{p})},
\end{equation*}
we get that 
\begin{equation*}
u(\theta_1,\theta_2)=\bigg(\frac{1}{|\theta_1-\theta_2|}\bigg)^{-2+\frac{1}{p}+\frac{1}{q}+2\varepsilon+\frac{\varepsilon}{1+\varepsilon}}=|\theta_1-\theta_2|^{2-\frac{1}{p}-\frac{1}{q}-\widetilde{\varepsilon}},    
\end{equation*}
where $\widetilde{\varepsilon}=2\varepsilon+\frac{\varepsilon}{1+\varepsilon}$. We observe that any sufficiently small $\widetilde{\varepsilon}$ can be expressed this way for a suitable choice of the parameter $\varepsilon$. The constant in inequality~\eqref{E:widetilde-T} then depends on $p_1$, $p_2$, $q$ and $\widetilde{\varepsilon}$. This completes the proof. 
\end{proof}

\section{Proof of Theorem~\ref{thm:main result}: Counterexample}\label{S:counterexamle n=1}

In this section, we prove the second part of Theorem~\ref{thm:main result}.

\begin{proof}[Proof of Theorem~\ref{thm:main result}, part \textup{(ii)}]
Throughout this proof, we set $\gamma=\frac{qp+q}{p+q}$.
For $i=1,2$, we consider the functions 
\begin{equation*}
f_i(x)=x^{-1/p_i}\big|\log{x}\big|^{-\gamma/p_i}\chi_{(0, \frac12)}(x), \quad x\in \R, 
\end{equation*}
and observe that $f_i\in L^{p_i}(\R)$. In addition, for $(\theta_1,\theta_2)\in \S^1$ we define
\begin{equation*}
\Omega(\theta_1,\theta_2)=\operatorname{sgn}(\theta_1) |\theta_1-\theta_2|^{\frac{1}{p}-2} |\log|\theta_1-\theta_2||^{-\frac{\gamma}{q}}\chi_{(-\frac{1}{2},\frac{1}{2})}(\theta_1-\theta_2).
\end{equation*}
We claim that $\Omega$ is an odd function satisfying $\Omega\in L^q(\S^{1},u^q)$, where $u(\theta_1,\theta_2)=|\theta_1-\theta_2|^{2-\frac{1}{p}-\frac{1}{q}}$. Indeed, by applying the change of coordinates $u_1=\theta_1-\theta_2$ and $u_2=\theta_1+\theta_2$, we have
\begin{align*}
\int_{\S^1}|\Omega(\theta_1,\theta_2)|^q &|\theta_1-\theta_2|^{2q-\frac{q}{p}-1} d\sigma(\theta_1,\theta_2)\\
&=2\int_{\sqrt2\S^1}|u_1|^{-1}|\log|u_1||^{-\gamma}\chi_{(-\frac12,\frac12)}(u_1)d\sigma(u_1,u_2)\\
&=8\sqrt{2}\int_{0}^{\frac{1}{2}}|v|^{-1}|\log|v||^{-\gamma}(2-|v|^2)^{-\frac{1}{2}}dv\\
&\leq 8 \sqrt{2} \int_{0}^{\frac{1}{2}}|v|^{-1}|\log|v||^{-\gamma}dv<\infty,
\end{align*}
where the second equality follows by using the co-area formula \cite[Appendix D]{GClassical}.

Let $x\in\R$ satisfy $x>10$. Then

\begin{equation*}
T_{\Omega}(f_1,f_2)(x)=\int_{(0,\frac{1}{2})^2} \left(\prod_{i=1}^2 y_i^{-1/p_i}\left|\log y_i \right|^{-\gamma/p_i}\right)
\frac{\Omega\bigg(\frac{(x-y_1,x-y_2)}{|(x-y_1,x-y_2)|}\bigg)}{|(x-y_1,x-y_2)|^2}dy_1 dy_2.
\end{equation*}
We observe that the integrand in the above integral is positive since $x-y_1$ is positive whenever $x>10$ and $y_1\in (0,1/2)$. Hence, we estimate $T_{\Omega}(f_1,f_2)(x)$ from below by the integral over the smaller region

\begin{equation*}
A = \left\{(y_1,y_2)\in\bigg(0,\frac{1}{2}\bigg)^2\,:\,2y_1 \leq y_2\leq 4y_1\right\}.
\end{equation*}
In this region, 
\[
\frac{|y|}{\sqrt{17}} \leq y_1\leq |y_1-y_2| \leq 3y_1 \leq 3|y|
\]
and
\[
x \leq |(x-y_1,x-y_2)| \leq \sqrt{2} x.
\]
Therefore
\begin{align*}
\Omega\bigg(\frac{(x-y_1,x-y_2)}{|(x-y_1,x-y_2)|}\bigg)&=\left|\frac{y_1- y_2}{|(x-y_1,x-y_2)|}\right|^{\frac{1}{p}-2} 
\left|\log{\left|\frac{y_1- y_2}{|(x-y_1,x-y_2)|}\right|}\right|^{-\frac{\gamma}{q}}\\
&\gtrsim x^{-\frac{1}{p}+2}|y|^{\frac{1}{p}-2}\left|\log\frac{|y|}{x}\right|^{-\frac{\gamma}{q}}.
\end{align*}
We remark that in the display above as well as in the remaining part of this proof, the constant in the inequality ``$\gtrsim$'' is allowed to depend on $p_1$, $p_2$ and $q$. 
After a change of variables, we obtain
\begin{align*}
T_{\Omega}(f_1,f_2)(x)&\gtrsim x^{-\frac{1}{p}}\int_{A}|y|^{-2}\left|\log|y|\right|^{-\frac{\gamma}{p}}\left|\log\frac{|y|}{x}\right|^{-\frac{\gamma}{q}}dy\\
&\gtrsim x^{-\frac{1}{p}}\int_{0}^{\frac{1}{2}}\left|\log r\right|^{-\frac{\gamma}{p}}\left|\log\frac{r}{x}\right|^{-\frac{\gamma}{q}}\frac{dr}{r}\\
&=x^{-\frac{1}{p}}\int_{\log 2}^{\infty}t^{-\frac{\gamma}{p}}\left|t+\log x\right|^{-\frac{\gamma}{q}}dt\\
&\gtrsim x^{-\frac{1}{p}}\int_{\log 2}^{\infty}(t+\log x)^{-\frac{\gamma}{p}-\frac{\gamma}{q}}dt
\gtrsim x^{-\frac{1}{p}}(\log x)^{(1-\frac{\gamma}{p}-\frac{\gamma}{q})}.
\end{align*}
Therefore,
\begin{equation*}
\Vert T_{\Omega}(f_1,f_2)\Vert_{L^{p}(\R)}^p\gtrsim\int_{10}^\infty x^{-1}(\log x)^{p(1-\frac{\gamma}{p}-\frac{\gamma}{q})}dx.
\end{equation*}
Thanks to our choice of $\gamma$, the above integral becomes infinite, i.e.\ $T_{\Omega}(f_1,f_2)\notin L^{p}(\R)$, as desired.
\end{proof}

\section{Proof of Theorem~\ref{T:counterexample}}\label{S:counterexamle-higher-n}

In this section, we provide counterexamples justifying the validity of Theorem~\ref{T:counterexample}.

\begin{proof}[Proof of Theorem~\ref{T:counterexample}]
\textup{(i)}
For a fixed combination of indices $p_1,p_2,p,q,n$ which satisfies the assumptions of the first part of the theorem, we will construct odd functions $\Omega$ supported in the specified set with $\|\Omega\|_{L^q(\mathbb S^{2n-1})}=1$ such that the norm of the operator $T_\Omega$ is arbitrarily large. The existence of the actual unbounded operator will then follow via usual summation arguments.

Throughout the first part of the proof, for a given $a\in \mathbb R^n$ and $R>0$, we denote by $B(a,R)$ the ball centered at $a$ with radius $R$. The constants in the relations $\approx$ and $\lesssim$ are allowed to depend on $p_1,p_2,q$ and $n$.

We take $\delta>0$ a  very small constant. We put $f_1=\chi_{B(-e_1,\delta)}$ and $f_2=\chi_{B(0,\delta)},$ where $e_1=(1,0,\dots,0)$ is the unit vector in $\mathbb R^n$. We denote $$A=\left\{\frac{(x+e_1,x)}{|(x+e_1,x)|}: ~x\in B\bigg(0,\frac{1}{100}\bigg)\subset\mathbb R^n\right\}$$ and 
$$B=\{\gamma\in \mathbb S^{2n-1}: {\rm dist}(\gamma,A)\leq 10\delta\}.$$ We put 
$\Omega= K(\chi_{B}-\chi_{-B}).$ We choose the constant $K$ such that $\|\Omega\|_{L^q(\mathbb S^{2n-1})}=1.$ It is easy to see that the $(2n-1)$-dimensional measure of the set $B$ in $\mathbb S^{2n-1}$ is $C\delta^{n-1},$ where $C$ depends on $n$, as the set $B$ is a slightly deformed tube with $n$ dimensions of length roughly $1/100$ and $n-1$ dimensions of roughly $10\delta.$ Thus, $K\approx\delta^{(1-n)/q}.$ Also, $\Omega$ is clearly odd and supported in the set $\{(\theta_1,\theta_2)\in \mathbb S^{2n-1}: ~|\theta_1-\theta_2|>1/2\}.$

Now, we see that $\|f_1\|_{L^{p_1}(\R^n)}\approx \delta^{n/p_1}$ and $\|f_2\|_{L^{p_2}(\R^n)}\approx \delta^{n/p_2}.$ Therefore 
$$
\|f_1\|_{L^{p_1}(\R^n)}\|f_2\|_{L^{p_2}(\R^n)}\approx \delta^{n/p}.
$$
Assume that $x\in B(0,\frac{1}{100})$ and $y_1$, $y_2$ are such that $f_1(x-y_1) \neq 0$ and $f_2(x-y_2)\neq 0$. Then $|(y_1,y_2)|$ is roughly $1$ and
\[
\frac{(y_1,y_2)}{|(y_1,y_2)|} \in B,
\]
i.e., the supports of $f_1$ and $f_2$ interact only with the positive part of the kernel in the integral~\eqref{eq:T-multilinear} defining the operator $T_\Omega$.
This means that we get $T_{\Omega}(f_1,f_2)(x)\approx K\|f_1\|_{L^1(\R^n)}\|f_2\|_{L^1(\R^n)}\approx K\delta^{2n}.$ Therefore 
$$\|T_{\Omega}(f_1,f_2)\|_{L^p(\R^n)}\gtrsim \left( \int_{B(0,\frac{1}{100})} K^p \delta^{2np} \right)^{1/p}\approx K \delta^{2n}.$$

Collecting these estimates we get
\begin{equation}\label{cexn}
 \|T_{\Omega}(f_1,f_2)\|_{L^p(\R^n)} \gtrsim \delta^{(1-n)/q}\delta^{2n} \delta^{-n/p} \|f_1\|_{L^{p_1}(\R^n)}\|f_2\|_{L^{p_2}(\R^n)}.
\end{equation}
We see that under the assumption~\eqref{prangebili}, $\frac{1-n}{q}+2n-\frac{n}{p}$ is negative and therefore the norm of $T_\Omega$ goes to infinity as $\delta$ goes to $0$.

The summation argument to produce an unbounded operator is in this case very simple. Namely, it is enough to consider the sequence $\delta_k=2^{-k},$ where $k\geq k_0$ for a sufficiently large integer $k_0$, then produce the related functions $\Omega_k$ and put $\Omega=\sum_{k=k_0}^\infty \frac{1}{k^2}\Omega_k$. As $q\geq 1,$ we have that the $L^q$-norm of $\Omega$ is finite. The sets where the functions $\Omega_k$ are positive are nested, therefore the estimate~\eqref{cexn} is preserved with $\delta$ replaced by any $\delta_k$, using the test functions $f_1$ and $f_2$ adapted to $\delta_k.$ Thus, the operator $T_\Omega$ is not bounded from $L^{p_1}(\R^n) \times L^{p_2}(\R^n)$ into $L^p(\R^n)$.

\textup{(ii)}
In order to prove the second part of the theorem, we fix $m\geq 3,$ $q>1$ and $p_1,\dots,p_m,p$ with $\frac 1 p = \sum_{i=1}^m \frac 1 {p_i}.$ We assume that~\eqref{prangemultili} holds. Throughout the remaining part of the proof, the constants in the relations $\approx$ and $\lesssim$ are allowed to depend on $m$, $q$ and $p_1,\dots,p_m$. 

Let 
\[
c_m=\frac{\sqrt{6}}{\sqrt{(m^2+m)(2m+1)}},
\]
then the point $(c_m,2c_m,\dots,mc_m)$ belongs to $\S^{m-1}$. 
We take $\delta>0$ a  very small constant. We put 
$$
f_1=\chi_{(-c_m-\delta,-c_m+\delta)}, ~f_2=\chi_{(-2c_m-\delta,-2c_m+\delta)}, ~\cdots~, f_m=\chi_{(-mc_m-\delta,-mc_m+\delta)}.
$$ 
We denote $$A=\left\{\frac{(x+c_m,x+2c_m,\dots,x+mc_m)}{|(x+c_m,x+2c_m,\dots,x+mc_m)|}\in\mathbb S^{m-1}:~ x\in \bigg(0,\frac{1}{100 \sqrt{m}}\bigg)\right\}$$ and 
$$B=\{\gamma\in \mathbb S^{m-1}:~ {\rm dist}(\gamma,A)\leq 10 \sqrt{m}\delta\}.$$ We put 
$\Omega= K(\chi_{B}-\chi_{-B}).$ We choose the constant $K$ such that $\|\Omega\|_{L^q(\S^{m-1})}=1.$ It is easy to see that the $(m-1)$-dimensional measure of the set $B$ in $\mathbb S^{m-1}$ is $C\delta^{m-2},$ where $C$ depends on $m$, as the set $B$ is a slightly deformed tube with one dimension of length roughly $1$ and $m-2$ dimensions of roughly $\delta.$ Thus, $K\approx\delta^{(2-m)/q}.$ Additionally, we readily verify that $\Omega$ is odd and supported in the set 
$$
\{(\theta_1,\theta_2,\dots,\theta_m)\in \mathbb S^{m-1}: |\theta_i-\theta_j|>m^{-3/2} \text{ for } i \neq j\}.
$$ 

Now, we see that for $i=1,\dots,m$ we have  $\|f\|_{L^{p_i}(\R)}\approx \delta^{1/p_i}.$ Therefore, $$
\|f_1\|_{L^{p_1}(\R)}\cdots\|f_m\|_{L^{p_m}(\R)}\approx \delta^{1/p}.
$$ 
Assume that $x\in (0,\frac{1}{100 \sqrt{m}})$ and $y_i$ is such that $f_i(x-y_i) \neq 0$ for $i=1,\dots,m$. Then 
\[
\frac{(y_1,\dots,y_m)}{|(y_1,\dots,y_m)|} \in B,
\]
and therefore
$T_{\Omega}(f_1,\dots,f_m)(x)\approx K\|f_1\|_{L^1(\R)}\cdots\|f_m\|_{L^1(\R)}\approx K\delta^{m}.$ Thus,
$$\|T_{\Omega}(f_1,\dots,f_m)\|_{L^p(\R)}\gtrsim \left( \int_{0}^{\frac{1}{100\sqrt{m}}}K^p \delta^{mp} \right)^{1/p}\approx K \delta^{m}.$$

Collecting these estimates we get
\begin{equation*}
 \|T_{\Omega}(f_1,\dots,f_m)\|_{L^p(\R)} \gtrsim \delta^{(2-m)/q}\delta^{m} \delta^{-1/p} \|f_1\|_{L^{p_1}(\R)}\cdots\|f_2\|_{L^{p_m}(\R)}.
\end{equation*}
We see that under the assumption~\eqref{prangemultili}, $\frac{2-m}{q}+m-\frac{1}{p}$ is negative and therefore the norm of $T_\Omega$ goes to infinity as $\delta$ goes to $0$. The summation argument is then similar to that performed in part \textup{(i)}.
\end{proof}

\Addresses

\end{document}